\documentclass[12pt]{article}
\usepackage{amssymb,amsmath}
\input{amssym.def}\input{amssym}
\usepackage{graphicx, tikz}
\usepackage{hyperref,url}

\def\subjclass#1{\par\medskip
\noindent\textbf{Mathematics Subject Classification (2010):} #1}
\def\keywords#1{\par\medskip
\noindent\textbf{Keywords.} #1}

\newcommand{\N}{{\mathbb N}}

\newcommand{\R}{{\mathbb R}}

\newcommand{\Z}{{\mathbb Z}}

\newcommand\cA{{\mathcal A}}
\newcommand\cB{{\mathcal B}}

\newcommand\cL{{\mathcal L}}

\newcommand\cR{{\mathcal R}}

\def\p{\mathfrak{p}}

\newcommand\gS{{\mathfrak S}}

\newtheorem{theorem}{Theorem}[section]
\newtheorem{prop}{Proposition}[section]
\newtheorem{corollary}{Corollary}[section]
\newtheorem{lemma}{Lemma}[section]
\newtheorem{remark}{Remark}[section]
\newtheorem{example}{Example}[section]
\newtheorem{definition}{Definition}[section]
\newenvironment{proof}{\noindent {\bf Proof.}}{ \hfill $\Box$\\ }
\newenvironment{proofof}[1]{\noindent {\bf Proof of #1.}}{ \hfill $\Box$\\ }

\def\orb{\operatorname{orb}}
\def\supp{\operatorname{supp}}

\def\htop{h_{top}}

\title{Speedups of linearly recurrent subshifts}
\author{Henk Bruin
\thanks{Faculty of Mathematics, University of Vienna,  Oskar Morgensternplatz 1, 1090 Vienna, Austria; {\it henk.bruin@univie.ac.at}.
}}
\date{\today}

\begin{document}

\maketitle

\begin{abstract}
A speedup, like a time change in discrete time dynamics, is a way of moving faster through the orbits of a dynamical system.
Linearly recurrence is a stronger form of minimality for subshifts,
shared by e.g.\ all primitive substitution shifts and Sturmian shifts associated with rotation numbers of bounded type.
We prove that the homeomorphic speedup of a linearly recurrent two-sided subshift is again linearly recurrent.
\end{abstract}

\subjclass{Primary: 37B10, % Symbolic dynamics
Secondary: 37B05  %	Dynamical systems involving transformations and group actions with special properties (minimality, distality, proximality, expansivity, etc.)
}
	\keywords{speedup, linear recurrence, subshift, word-complexity.}

\section{Introduction}\label{sec:intro}

A speedup of a discrete time dynamical system
$(X,T)$ is another dynamical system that emerges by going
faster through the $T$-orbits. Namely, given a {\em jump-function} $p:X \to \N = \{1,2,3,\dots\}$,  the associated {\em speedup} is $S:X \to X$, $S(x) = T^{p(x)}(x)$.
Speedups can be viewed as a kind of discrete time analogue of time-changes
for continuous time flows, cf.\ \cite{KS87}.
The jump-function is a weaker form of the orbit cocycle of orbit equivalence, see e.g.\ \cite[page 23]{Bruin23}.
Speedups were considered from a measure-theoretic point of view in \cite{AOW85}
(its relation to cutting-and-stacking) and \cite{GPS95} (dimension groups for speedups).
More topological approaches were given in \cite{AAO18,Ash,AO24}.

Throughout this paper, we assume that $X$ is a Cantor set and $T: X \to X$ is continuous and invertible. We assume that the jump-function is continuous, injective (i.e., $T^{p(x)}(x) \neq T^{p(y)}(y)$ whenever $x \neq y$) and is such that $S$ is again invertible.
Since the space $X$ is compact, this implies in particular that $p$ is bounded, so that
\begin{equation}\label{eq:pmax}
p_{\max} := \max\{ p(x) : x \in X\}\ \text{ exists,}
\end{equation}
and that every $T$-orbit falls apart in a finite number of $S$-orbits.
The maximum $c$ of these numbers of $S$-orbits is called the {\em orbit number} and $c \leq p_{\max}$.
Ash~\cite[Proposition 4.1.2]{Ash} (and also  \cite[Proposition 2.2]{AAO18}) tells us that if $(X,T)$ is a minimal Cantor system with a continuous speedup $S = T^p$, then $p$ is continuous if and only if it is bounded.

A general question is which properties the speedup $(X,S)$ inherits from $(X,T)$.
Transitivity and minimality are not inherited in general. To illustrate this, assume that $\{X_1, X_2\}$ is a non-trivial clopen partition of the Cantor set $X$. If $p(x)$ is the first return time to the partition element that $x$ belongs to, then $S = T^p$ has $X_1$ and $X_2$ as invariant subsets. For subshifts, the extreme case of this is $p(x) = \min\{ n \geq 1 : x_n = x_0\}$.
 For the (one-sided) Fibonacci substitution shift ($0 \to 1, 1 \to 10$), with fixed point
 $$
 \rho = 0\ 1\ 0\ 01\ 010\ 01001\ 01001010\ 0100101001001 \ldots\ ,
$$
the speedup with the above jump function
uses the return words $A = 0$, $b=01$ of $0$, and $C = 10$, $D=100$ as alphabet.
The sequence $\rho$ falls apart in $c=2$ orbits
of the speedup:
$$
ABAABABAABAAB \ldots \quad \text{ and } \quad
CDCCDCDCCDCCD  \ldots
$$
That is, as the relevant theory \cite{Dur98, HZ98} (and also~\cite{AAO18, AO24}) predicts, the speedup renders two copies of the Fibonacci substitution shift, precisely  that are disjoint.

For this reason, we will assume throughout that our speedups are transitive.
A simple non-constant, non-first-return mechanism of the jump-function that retains transitivity, e.g.\ in the two-sided Fibonacci substitution shift, is
% $$
% p(x) = \begin{cases}
%         1 & x \in [01.110],\\
%         3 & x \in [0.1110],\\
%         2 & \text{otherwise,}
%        \end{cases}
%  $$
\\
\begin{figure}[h]
\begin{center}
\begin{tikzpicture}[scale=1]
\node at (-4,0)
{$p(x) = \begin{cases}
        1 & x \in [01.110],\\
        3 & x \in [0.1110],\\
        2 & \text{otherwise,}
       \end{cases}$};
\draw[->] (0.1,0.1) .. controls (1,0.4) .. (1.9,0.1); \node at (1,0) {\small $\bullet$}; \node at (1,-0.4) {\small $0$};
\draw[->] (1.1,0.1) .. controls (2,0.4) .. (2.9,0.1); \node at (2,0) {\small $\bullet$};  \node at (2,-0.4) {\small $1$};
\draw[->] (2.1,0.1) .. controls (3.5,0.6) .. (4.9,0.1); \node at (3,0) {\small $\bullet$};  \node at (3,-0.4) {\small $1$};
\draw[->] (3.1,0.1) .. controls (3.5,0.3) .. (3.9,0.1); \node at (4,0) {\small $\bullet$};  \node at (4,-0.4) {\small $1$};
\draw[->] (4.1,0.1) .. controls (5,0.4) .. (5.9,0.1); \node at (5,0) {\small $\bullet$};   \node at (5,-0.4) {\small $0$};
\draw[->] (5.1,0.1) .. controls (6,0.4) .. (6.9,0.1); \node at (6,0) {\small $\bullet$};  \node at (6,-0.4) {\small $0$};
\end{tikzpicture}
\end{center}
\end{figure}
\\
%\noindent
where the dot stands directly after the zero-th position in the words of the cylinder sets. The figure indicates how $S$-orbits go through $\sigma$-orbits.
The transitive speedup of a minimal system is again minimal, as we show in Proposition~\ref{prop:min}.

The speedup acts on the same space as the original map, but in case of a subshift $(X,\sigma)$ on a finite alphabet $\cA$,  to interpret the speedup again as
a subshift,
we use the collection of strings
\begin{equation}\label{eq:cB}
\cB = \{ x_0x_1 \dots x_{p(x)-1} : x \in X\}
\end{equation}
as alphabet; that is, each $x \in X$, is a concatenation of words in $\cB$
and $S$ is the left-shift on the level of these concatenation words.
Note that each $x \in X$ can be decomposed in at most $c$ different and completely disjoint ways, which we refer to as the different ways that $S$ {\em passes} through $x$.

It is easy to show that speedups of subshifts of finite type (SFTs) are again SFTs, and speedups of sofic shifts are sofic, see Proposition~\ref{prop:SFT}.
Ash et al.\ \cite{AAO18, AO24} showed, among other results, that minimal speedups of substitution shifts are substitution shifts, and also that the minimal speedup of odometers are odometers.
However, speedups of Toeplitz shifts need not be Toeplitz shifts, see \cite{AR25}.

\begin{definition}\label{def:LR}
 Let $(X, \sigma)$ be a subshift, where $\sigma$ denotes the left-shift,
 and $\cL(X)$ is the language of $X$.
We say that $(X,\sigma)$ is {\em linearly recurrent} if there is $L \in \N$ such that
every finite word $w\in \cL(X)$ reappears in every $x \in X$ with gap $\leq L |w|$ shifts.
\end{definition}

Linearly recurrent shifts are minimal, have a unique shift-invariant probability measure, $\mu$, and $\mu([u]) \geq 1/(L|u|)$, see \cite{Dur00}.
The aim of this paper is to prove the following theorem.

\begin{theorem}\label{thm:main}
Let $S = \sigma^p$ be the homeomorphic transitive speedup of a two-sided subshift $(X,\sigma)$.
Then $\sigma$ is linearly recurrent if and only if $S$ is linearly recurrent.
\end{theorem}

\noindent
Only for two-sided shifts, a speedup can be homeomorphic, but by passing from a one-sided to the two-sided shift with the same language, one can easily extend Theorem~\ref{thm:main} to one-sided shifts.

 The next section gives the relevant definitions, some related results, and estimates the word-complexity of speedups.  Section~\ref{sec:skew} investigates the structure of
non-abelian finite group-extensions and gives results on essential values and minimality needed for the main theorem.
Finally, in Section~\ref{sec:LR} we prove the main theorem.

\section{Preliminaries and miscellaneous results}\label{sec:prelim}

As mentioned in the introduction, speedups of transitive (Cantor) systems
need not be transitive, but we have the following:

\begin{prop}\label{prop:min}
 The continuous homeomorphic transitive speedup of a
 minimal system is minimal.
\end{prop}

\begin{proof}
Assume by contradiction that the speed-up, denoted $(X,S)$, is not minimal, so there is an $S$-invariant set $Y \subsetneq X$. Let $N = \max_{y \in Y} p(y)$, then $Y' = \bigcup_{k=0}^{N-1}\sigma^{-k}(Y)
=  \bigcup_{k=0}^{N-1}\sigma^{-k}(Y)$ is a shift-invariant set, and by minimality of $\sigma$, $Y' = X$.
Let $y \in X$ have a dense $S$-orbit.
If the interior $\mathring{Y} \neq \emptyset$, then there is $n \geq 0$ such that $S^n(y) \in \mathring{Y}$. In this case, the corresponding $\omega$-limit set $\omega_S(y) \subset Y \neq X$, contradicting that $\orb_S(y)$ is dense in $X$.
Alternatively, $\mathring{Y} = \emptyset$ but then $X = \bigcup_{k=0}^{N-1}\sigma^{-k}(Y)$
has empty interior too. This contradiction proves the lemma.
\end{proof}

\subsection{Subshifts of finite type and sofic shifts}
Let $\cA$ be a finite alphabet, and equip the
left-shift on the space $\cA^{\Z}$ of two-sided infinite sequences
with product topology.
Let $\sigma:\cA^\Z \to \cA^\Z$, $\sigma(x)_i = x_{i+1}$,
be the left-shift.  We say that $(X,\sigma)$
is a subshift if $X$ is
$\sigma$-invariant (in fact, $\sigma(X) = X$) closed non-empty subset of $\cA^{\Z}$.
Let $\cL(X) := \{ w \in \cA^n : n \in \N, w \text{ appears as a subword in some } x \in X \}$ be the
{\em language} of the subshift. Formally, also the empty word $\epsilon$ belongs to $\cL(X)$, but we don't need it, and including it only makes formulas more cumbersome.

A {\em subshift of finite type} (SFT) $(X,\sigma)$ on a finite alphabet $\cA$ is characterized a {\bf vertex-labeled} transition graph
$G= (V,E)$, where the vertex set $V = \cA$ and the directed edge set $E$ is unlabeled.
Infinite word paths in this graph correspond bijectively to the elements $x \in X$, via the labeling of the vertices.

A {\em sofic shift} $(X,\sigma)$ on a finite alphabet $\cA$ is characterized by an {\bf edged-labeled} transition graph
$G=(V,E)$. The vertex set is unlabeled but each (directed) $e \in E$ gets a label from $\cA$, where (contrary to SFTs) the same label can be assigned to  multiple edges. Infinite word paths in this graph correspond bijectively to the elements $x \in X$, via the labeling of the edges.

\begin{prop}\label{prop:SFT}
 The injective continuous speedup of a subshift of finite type is a subshift of finite type.
 The injective continuous speedup of a sofic shift is sofic.
\end{prop}

\begin{proof}
Let $N$ be minimal such that $p$ is constant on every centered $2N+1$-cylinder. Set $M = 2\max\{ p_{\max}, N\}+1$ and consider the alphabet $\tilde\cA = \cA^M$.
The SFT $(X,\sigma)$ over $\cA$ can be represented as an $M$-block-code shift over $\tilde\cA$, with its vertex-labeled transition graph $G= (V = \tilde\cA, E)$. Speeding up $\sigma$ results into a transition graph
$G' = (V' , E')$ where $V' = V = \tilde\cA$ and for each $v \in V$, the  edges $e \in E'$ outgoing from $v$ are the paths of length $p(v)$ starting at $v$.
That is, $(X,S = \sigma^p)$ has a vertex-labeled transition graph, so it is an SFT.

The proof for sofic shifts is the same, only with edge-labeled transition graphs.
\end{proof}

\subsection{Word-complexity of speedups}
In the above notation of subshifts, linear recurrence means that for every $w \in \cL(X)$ and $x \in X$ and $i \in \Z$, there is $k \in \{0, 1, \dots , L|w|\}$ such that $x_{i+k+1} \dots x_{i+k+|w|} = w$.
Linear recurrence is preserved by topological conjugacy, although the constant $L$ can change.
Linearly recurrent shifts are automatically minimal and in fact
uniquely ergodic.
%Let $\mu$ denote the corresponding shift-invariant probability measure.

Let $\p_\sigma(n) = \# \{ w \in \cL(X)  : |w| = n\}$
be the {\em word-complexity} of $(X,\sigma)$.
Linearly recurrent shifts have sublinear complexity: $\p_{\sigma}(n) \leq Ln$.
If $(X,\sigma)$ is linearly recurrent with constant $L$, then its word-coplexity $\p_{\sigma}(n) \leq Ln$ as well, but there are many subshifts of linear word-complexity that are not linearly recurrent,
e.g.\ all Sturmian subshifts associated to a rotation number of unbounded type. So the following proposition is, for linear recurrent shifts, a direct consequence of Theorem~\ref{thm:main}, but it works much more generally and has a much easier proof than Theorem~\ref{thm:main}.

\begin{prop}\label{prop:WC}
 Given a homeomorphic speedup $S=\sigma^p$ of a subshift $(X,\sigma)$, there is a constant $K$ such that the word-complexity of $S$ satisfies
 $\p_S(n) \leq K \, \p_\sigma(p_{\max} \, n)$ for $p_{\max} = \max\{ p(x) : x\in X\}$
 from \eqref{eq:pmax}.
\end{prop}

\begin{proof}
 Since the jump-function is continuous and the space $X$ is compact, there is $N \in \N$ such that $p$ is constant on each $N$-cylinder.
 
 Let $Z = [z_1\dots z_n] \subset X$ be any $n$-cylinder.
 If  a sped-up orbit visits position $k \in [N+1, n-N]$ of this cylinder set,
 then the other positions in $[N+1, n-N]$ are fixed as well.
 We make no such uniqueness claim for positions in $[1,N]$ and $[n-N+1,n]$, but these are at most $\#\cA^{2N}$ patterns of positions.
 Hence, associated to $Z$, there are at most $p_{\max} \#\cA^{2N}$ $n/p_{\max}$-cylinders.
 This gives $\p_S(n) \leq K\, \p_\sigma(p_{\max} n)$ for $K = \#\cA^{2N} p_{\max}$
 where $p_{\max}$ is from \eqref{eq:pmax}.
\end{proof}

It follows that if $(X,\sigma)$ has linear, quadratic, cubic, etc., or even exponential word-complexity, so does the speedup.
For positive entropy shifts, this result just reproduces Abramov's formula
$\htop(Y,S) = (\int_X p \, d\mu)  \htop(X,\sigma)$ as in \cite{N69} and \cite[Proposition 2.9]{AAO18}.

\section{Finite non-abelian group extensions}\label{sec:skew}
Let $T:X \to X$ be a minimal homeomorphism on a Cantor set.
For a finite\footnote{Much of this section also holds for compact groups, but as we only apply it to a group of permutations in Section~\ref{sec:LR}, we restrict ourselves to finite groups.} group $G$, consider a skew-product
$$
F: X \times G \to X\times G, \qquad (x,g) \mapsto (T(x), g \cdot \phi(x) )
$$
for a continuous skew-function $\phi:X \to G$.
Write
\begin{equation}\label{eq:phin}
\phi^n(x) =
\begin{cases}
\phi(x) \cdot \phi(Tx) \cdots \phi(T^{n-1}(x) )  & \text{ if } n > 0,\\[1mm]
e \quad \text{(the identity of $G$)} & \text{ if } n = 0,\\[1mm]
\left( \phi(T^nx) \cdot \phi(T^{n+1}x) \cdots \phi(T^{-1}x) \right)^{-1}  & \text{ if } n < 0,
\end{cases}
\end{equation}
where the final ${}^{-1}$ in the third line refers to the inverse in the group.

Define for each $x \in X$ the set
\begin{equation}\label{eq:Gx}
G_x := \{ g \in G : \exists n_k \nearrow \infty \text{ such that }
\phi^{n_k}(x) \equiv g \text{ and } T^{n_k}(x) \to x\}.
\end{equation}
If $g, g' \in G_x$, say with corresponding sequence $(n_k)_{k \in \N}$ and
$(n'_l)_{l \in \N}$,
then, by continuity of $\phi^{n_k}$,
\begin{equation}\label{eq:G1}
\phi^{n_k+n'_l}(x) =  g' \cdot \phi^{n_k}( T^{n'_l}(x) ) \to g' \cdot g,
\quad \text{ as } l \to \infty.
\end{equation}
This holds for every $k$. Hence, by continuity of each $T^{n_k}$, there is a sequence $k_l \nearrow \infty$ such that
$T^{n_{k_l} + n'_l}(x) \to x$. This shows that $G_x$ is closed under composition.

\begin{remark}
One can just as well argue that, by continuity of $\phi^{n'_l}$,
$$
\phi^{n_k+n'_l}(x) = \phi^{n'_l+n_k}(x) =  g \cdot \phi^{n'_l}( T^{n_k}(x) ) \to g \cdot g',
\quad \text{ as } l \to \infty.
$$
However, this still doesn't imply that $G_x$ is abelian, because $(n_k+n'_l)_{l \in \N, k \in K}$
and $(n'_l+n_k)_{k \in \N, l \in L}$ might be entirely different sequences
for infinite sets $K, L \subset \N$, with only finite overlap.
Letting $k \to \infty$ and $l \to \infty$ at different speeds can still lead to different limit points in $G$.
It seems that if $T$ is equicontinuous, then $G_x$ is an abelian subgroup.
\end{remark}

Because $G$ is a finite group, every $g \in G$ has finite order, say order $d$, so $g^d = e$ is the identity. Applying \eqref{eq:G1} $d-2$ and $d-1$ times with $g' = g$, we obtain that $g^{d-1} = g^{-1} \in G_x$ and $g^d = e \in G_x$.
This proves for each $x \in X$ that $G_x$ is a subgroup of $G$.
%Using the same finiteness, and the fact that $g^{d-1} = g^{-1}$,
% we see that $G_x$ is in fact a subgroup of $G$.
%
% \begin{remark}
% Suppose that $g \in G_x$ with associated sequence $(n_k)_{k \in \N}$. Then
%  $$
% \phi^{n_k-n_l}(x) =
% \phi^{-n_l+n_k}(x) = \phi^{-n_l}( T^{n_k}(x) ) \cdot g \to g^{-1} \cdot g = e,
% \qquad \text{ as } k \to \infty.
% $$
% This holds for every $l$, so we can find a sequence $l_k \nearrow \infty$
% so that $T^{n_k-n_{l_k}}(x) \to x$.
%
% Then also, for every $l \in \N$ and the above sequence $(l_k)_{k \in \N}$,
%  $$
% \phi^{n_k-n_l-n_{l_k}}(x) =
% \phi^{-n_l+(n_k-n_{l_k}) }(x) = \phi^{-n_l}( T^{n_k-n_{l_k}}(x) ) \cdot e \to g^{-1},
% \qquad \text{ as } k \to \infty.
% $$
% This shows that finiteness of the group is not essential to prove that
% $G_x$ is a subgroup of $G$.
% \end{remark}

\begin{lemma}\label{lem:Gxs_conj}
Let $x,y \in X$ be two points in the same $T$-orbit.
Then $G_x$ and $G_y$ are conjugate subgroups of $G$,
 i.e., there is $g$ such that $G_y = g \cdot G_x \cdot g^{-1}$.
We can in fact take $g = \phi^n(y)$ where $n \in \Z$ is such that $T^n(y) = x$.
\end{lemma}

\begin{proof}
Take
$n \in \Z$ such that $T^{n_k}(y) \to x$ and
 $g = \phi^n(y)$.
 Now for $h \in G_x$ arbitrary, take $n'_l \nearrow \infty$ such that $T^{n'_l}(x) \to x$ and $\phi^{n'_l}(z) \equiv h$.
 Then
 \begin{eqnarray*}
  \phi^{-n+n'_l+n}(y) &=& \phi^n(y) \cdot \phi^{n'_l}(x)  \cdot \phi^{-n}(T^{n'_l}(x))
   = g \cdot h  \cdot \phi^{-n}(T^{n'_l}(x)) \\
  &\to_{l \to \infty}& g \cdot h  \cdot \phi^{-n}(x)
  = g \cdot h \cdot g^{-1},
\end{eqnarray*}
because $\phi^{-n}(x) = \phi^{-n}(T^ny) = (\phi^n(y))^{-1} = g^{-1}$ by \eqref{eq:phin}.
Since this holds for all $h \in G_x$, we have $G_y \supset g \cdot G_x \cdot  g^{-1}$, and in particular, $\# G_y \geq \# G_x$.
Reversing the roles of $x$ and $y$ gives the required equality.
\end{proof}

\begin{prop}\label{prop:cont}
Assume that $T: X \to X$ is minimal.
 The map $x \mapsto G_x$ is piecewise constant.
\end{prop}

\begin{proof}
Let $x \in X$ be ``minimal'' in the sense that for no $y \in X$, $G_y$ is a proper subgroup of $H := G_x$.
For each $g \in G$, we can try to find a sequence
$(r^g_k)_{k \in \N}$ such that $\phi^{r^g_k}(x) \equiv g$ and $d(T^{r^g_{k+1}}(x),x) \leq \frac12 d(T^{r^g_k}(x),x)$ for all $k$.
If $g \notin H$, then such a sequence is necessarily finite, so there is a largest integer $N_g$ in this sequence.
(If $\phi^n(x) \neq g$ for all $n \geq 0$, then we set $N_g = 0$.)

Let $N = \max \{ N_g : g \notin H\}$ and let $U_x$ be a clopen neighborhood of $x$ (in fact, a cylinder set if $(X,T)$ is a subshift),
such that $T^n(x) \notin U_x$ for all $1 \leq n \leq N$.
It follows that if $T^n(x) \in U_x$, then
$\phi^n(x) \in H$.
Let $V \subset U_x$ be any maximal clopen set such that its return time $r_V$ to $U_x$ is constant and
$\phi^{r_V}|_V \equiv g_V$ is constant, as well.
By minimality, there is $n \in \N$ such that $T^n(x) \in V$, so $g := \phi^n(x) \in H$.
But then $T^{n+r_V}(x) \in U_x$, so that also
$\phi^{n+r_V}(x) = g \cdot g_V \in H$, so $g_V \in H$.

Since $g_V \in H$ for every return domain $V$ of $U_x$,
it follows $\phi^n(y) \in H$ whenever $y, T^n(y) \in U_x$, and therefore $G_y \subset G_x = H$.
But the choice of $x$ at the beginning of the proof means that
$G_y = G_x$ for all $y \in U_x$.
Therefore $y \in X_H$ for all $y \in U_x$,
so $x$ is an interior point of $X_H$, and because $x \in X_H$ was arbitrary, $X_H$ is open.

By minimality, there is $\ell \in \N$ such that $X = \bigcup_{i=0}^{\ell-1} U_x$,
and by Lemma~\ref{lem:Gxs_conj}, $G_z = g \cdot H \cdot g^{-1}$ for each $z \in T^i(U_x \cap A_i)$, $i = 0, \dots, \ell-1$, where $A_i$ is any set on which $\phi^i \equiv g$ is constant.
This proves that $x \mapsto G_x$ is piecewise constant.
\end{proof}

\begin{corollary}\label{cor:Gxs_conj}
Let $T:X \to X$ be a minimal Cantor system, and $x,y \in X$ arbitrary.
Then $G_x$ and $G_y$ are conjugate subgroups of $G$,
 i.e., there is $g$ such that $G_y = g \cdot G_y \cdot g^{-1}$.
\end{corollary}

\begin{proof}
This follows immediately from Lemma~\ref{lem:Gxs_conj} and the continuity of
$x \mapsto G_x$.
In detail: by minimality, we can choose
$n_k \nearrow \infty$ such that $T^{n_k}(y) \to x$ and
since $G$ is finite, we can assume that $\phi^{n_k}(y) \equiv g$.
 Now for $h \in G_x$ arbitrary, take $k$ so large, i.e., $z := T^{n_k}(y)$ so close to $x$ that $h \in G_z$.
Next take $n'_l \nearrow \infty$ such that $T^{n'_l}(z) \to z$ and $\phi^{n'_l}(z) \equiv h$.
 Then
 \begin{eqnarray*}
  \phi^{-n_{k'}+n'_l+n_k}(y) &=& \phi^{n_k}(y) \cdot \phi^{n'_l}(z)  \cdot \phi^{-n_k}(T^{n'_l}(z))
   = g \cdot h  \cdot \phi^{-n_k}(T^{n'_l}(z)) \\
  &\to_{l \to \infty}& g \cdot h  \cdot \phi^{-n_k}(z)
  = g \cdot h \cdot g^{-1},
\end{eqnarray*}
because $\phi^{-n_k}(z) = \phi^{-n_k}(T^{n_k}y) = (\phi^{n_k}(y))^{-1} = g^{-1}$ by \eqref{eq:phin}.
Since this holds for all $h \in G_x$, we have $G_y \supset g \cdot G_x \cdot  g^{-1}$, and in particular, $\# G_y \geq \# G_x$.
Reversing the roles of $x$ and $y$ gives the required equality.
\end{proof}

The sets $G_x$ play the role of {\em local essential values}.
We say that $h \in G$ is an {\em essential value} if for every $A \subset X$
with $\mu(A) > 0$ there is $n \in \Z$ such that
$$
\mu\left( A \cap T^{-n}A \cap \{ x \in X : \phi^n(x) = h \}
\cap \{ x \in X : T^n(x) \neq x\} ) \right) > 0.
$$
For abelian groups $G$, $G_x$ coincides with the set of essential values, but for non-abelian groups, the essential values may not be more than the centralizer of $G$, or the intersection $\bigcap_{x \in X} G_x$, and thus $G$ is of less use.

The collection of essential values forms a subgroup of $G$, see \cite{S77}.

\begin{lemma}\label{lem:ergod}
Let $T: X \to X$ be minimal with an ergodic $T$-invariant measure $\mu$.
Let $F(x,g) = (T(x), g \cdot \phi(x))$ be a continuous skew-product, for a finite group $G$.
 Then there is $x \in X$, a neighbourhood $U$ of $x$ and an ergodic $F$-invariant measure $\rho$ such that
 $\rho|_{U \times H}$ is a multiple of $\mu|_U \otimes \nu_H$,
 where $H = G_x$ and $\nu_H$ is right Haar measure (i.e., counting measure) on $H$.
\end{lemma}

\begin{proof}
The proof of Proposition~\ref{prop:cont} implies that there is $x \in X$ and a neighbourhood $U = U_x$ of $x$ such that if
$T^n(y) \in U$ for some $y \in U$, then $\phi^n(y) \in H$.
Let $T_U = T^{\tau}$ be the first return map to $U$, where $\tau(y) = \min\{ n \geq 1 : T^n(y) \in U\}$ is the first return time of $y \in U$ to $U$.
Then $\mu_U := \frac{1}{\mu(U)} \mu|_U$ is $T_U$-invariant and ergodic because $\mu$ is  $T$-invariant
and ergodic.
Let
$$
F_U: U \times H \to U \times H, \qquad
(y\, ,\, g) \mapsto (T_U(y) \ , \ g \cdot \phi^{\tau(y)}(y) )
$$
be the corresponding skew-product; it preserves $\mu_U \otimes \nu_H$.

To show ergodicity, assume that $Y \subset U \times H$ is invariant under the skew-product and $\mu_U \otimes \nu_H(Y) > 0$.
Since $\mu_U$ is ergodic, the projection of $Y$ onto $U$ has full $\mu_U$-measure.

Assume now by contradiction that $\mu_U \otimes \nu_H(Y^c) > 0$. Both $Y$ and $Y^c = (U \times H) \setminus Y$ are invariant, so their projections are $T_U$-invariant sets of $U$ of positive measure, and by ergodicity of $\mu$, these projections are $U \bmod \mu_U$.
Therefore there is a set $A \subset U$ and $h_1, h_2 \in H$ such that
\begin{equation}\label{eq:h1h2}
\mu \otimes \nu_H(Y \cap (A \times \{ h_1\})) > 0 \quad \text{ and } \quad \mu \otimes \nu_H(Y \cap (A \times \{ h_2\})) = 0.
\end{equation}
(We could write $T$ here instead of $T_U$ because $T^n(y) \in U$ for $y \in U$ is equivalent to $T^n(y) = T^{n'}_U(y)$ for some $|n'| \leq |n|$.)
Take $h = h_1^{-1} \cdot h_2 \in H$, so by definition of essential value, there is $n \in \Z$ such that
$\mu(B) > 0$ for $B := A \cap T^{-n}(A) \cap \{ x \in X : \phi^n(x) = h, x \neq T^n(x)\}$.
But then
$$
\mu \otimes \nu_H( \{ (T^n(x), h_1 \cdot \phi^n(x)) : x \in B\} ) > 0,
$$
but
$$
\{ (T^n(x)\, ,\, h_1 \cdot \phi^n(x)) : x \in B\} \subset A \times \{ h_2\},
$$
contradicting the choice of $A$ and $h_2$
in \eqref{eq:h1h2}. Hence $\mu_U \otimes \nu_H$ is
ergodic.

Finally, using Kac' formula, we spread this mass over the whole
skew-product space
to obtain the required $F$-invariant measure $\rho$:
$$
\rho(\, \cdot\, ) := \frac{1}{\Lambda} \sum_{n \geq 1}
\mu_U \otimes \nu_H \left( F^{-n}(\, \cdot\, ) \cap \{ (y,g) \in U \times H : n < \tau(y) \} \right),
$$
for normalising constant $\Lambda := \sum_{n \geq 1}
n \mu_U(\{ y : \tau(y) = n\})$.
This formula gives $\rho|_{U \times H} = \frac{1}{\Lambda}  \mu_U \otimes \nu_H$, and the ergodicity of $\rho$ follows from the ergodicity of
$\mu_U \otimes \nu_H$.
\end{proof}

The following result is in essence due to Furstenberg \cite[Section 2 and Theorem 4.1]{Fu61}, who proved it for toral automorphisms.
The idea was used to prove the unique ergodicity (and minimality) of arbitrary
group extensions over uniquely ergodic bases.
\\

\begin{prop}\label{prop:Furst}
 Let $(X,T,\mu)$ be uniquely ergodic
 and $G$ a compact group with (right) Haar measure $\nu_G$.
 Let the group extension $S:Y \to Y$ defined on $Y := X \times G$ as
 $F(x,g) = (T(x), g\phi(x))$ for some $\phi:X \to G$.
 If $F$ is ergodic w.r.t.\ $\nu = \mu \otimes \nu_G$, then
 $F$ is uniquely ergodic.

 If in addition, $T$ is minimal, then $S$ is minimal, as well.
\end{prop}

\begin{proof}
Let $(x,g) \in Y$ be the $\mu \otimes \nu_G$-typical point, so it satisfies Birkhoff's Ergodic Theorem w.r.t.\ {\bf every} continuous function $\varphi:Y \to \R$.
For any $h \in G$, $\varphi_h$ defined by $\varphi_h(g) = \varphi(hg)$ is continuous too,
so $(x,hg)$ is $\nu$-typical w.r.t.\ $\varphi$ because $(x,g)$ is $\nu$-typical
w.r.t.\ $\varphi_h$.
It follows that there is a subset $W \subset X$ with $\mu(W) = 1$
such that $W \times G$ consists entirely of $\nu$-typical points.

If $\nu_{\gS'}$ was another ergodic $S$-invariant probability measure,
then the argument above gives a set $W' \subset X \setminus W$
with $\nu_{\gS'}(W) = 1$ such that
$W' \times G$ consist entirely of $\nu_{\gS'}$-typical points.
Then the projected measure $\mu'$ on $X$ defined by
$\mu'(A) = \nu_{\gS'}(A \times G)$ is $T$-invariant, and satisfies
$\mu'(W') = 1$.
But $W$ and $W'$ are disjoint, so $\mu \neq \mu'$, contradicting that $T$ is uniquely ergodic.

Now if $T$ is also minimal, then the support $\supp(\mu) = X$, and
$\supp(\nu) = \supp(\mu) \times \supp(\nu_G)
= X \times G = Y$. Combined with unique ergodicity of $\nu$,
this implies that $S$ cannot have closed invariant proper subsets, so it is minimal.
\end{proof}

\section{Linear recurrence of speedups}\label{sec:LR}

A major tool in the proof of Theorem~\ref{thm:main}
are return words.

\begin{definition}
Let $(X,\sigma)$ be a subshift and $w \in \cL(X)$.
 We call $R$ a {\em return word} of $w$ if
 $Rw \in \cL(X)$, begins and ends with $w$, but contains no other occurrences of $w$.
\end{definition}

 Linear recurrence implies that each return word $R$ has length $|R| \leq L|w|$, reoccurs within $L|R| \leq L^2|w|$ iterates, and there are at most
 $L (L+1)^2$ return words, see \cite{DHS99}. Let $\cR_w$ denote the collection of return words.
 Each $x \in X$ can be uniquely written as concatenation of return words
 \begin{equation}\label{eq:xR}
 x = (R_i)_{i \in \Z} = \dots R_{-2}R_{-1} R_0 R_1 R_2 \dots,
 \qquad R_i \in \cR_w.
 \end{equation}
 The collection of $x \in X$ written as concatenation of letters in the alphabet $\cR_w$
 gives the {\em derived shift} $(X_w,\sigma)$, which is thus a linearly recurrent shift in its own right.

 \begin{remark}\label{rem:c+s}
  An equivalent approach equivalent to return words is by means of cutting-and-stacking. The return words $R \in \cR_w$ are then stacks and the letters in $\cA$ that $R$ are composed of can be read
  off on the stack, bottom to top.
 \end{remark}

An outline of the proof of why linear recurrence of $(X,\sigma)$ implies linear recurrence of the speedup starts as follows.
Take $w$ so long that the return words contain the essential information
of the $S$-orbits that pass through it.
That is, the $c$ different $S$-orbits that $\orb(x)$ splits into, partitions the letters of $R \in \cR_w$ into $c$ subsets, characterized by so-called {\em entry positions}.
For each pair of concatenated return words $RR'$, the transit of the $S$-orbit, starting at the $j$-th entry position in $R$, uniquely selects an entry position in $R'$, say $j'$.
Doing this for all $j \in \{1, \dots, c\}$, we get a permutation $\psi_{RR'}: \{1, \dots, c\} \to \{1, \dots, c\}$.
This leads to a group extension, i.e., skew-product of $X_w \times \gS$  and $\gS$ is the group of permutations of $\{1, \dots, c\}$.
This is the first step of the proof, which we formalize in the following construction.
\\[3mm]
{\bf Construction of the group extension:}
 Let $(X,\sigma)$ be a two-sided shift, and
 let $S = \sigma^p$ be the speedup, where the jump-function $p:X \to \N$ is continuous.
 Since $X$ is compact,  there is $K \in \N$ so that $p$ is constant on cylinders with
 anticipation and history $K$ (i.e., centered $2K+1$-cylinders).
Assume that $w \in \cL(X)$ is so long that $|R| \geq p_{\max} + 2K+2$
for each $R \in \cR_w$, where $p_{\max}$ as in~\eqref{eq:pmax}).

Recall that there are $c$ different $S$-orbits passing through $w$.
 Inside $w$, we select $c$ indices  $K+1 \leq j \leq K+p_{\max}$, which will be referred to as the fixed {\em entry positions},
 If the entry position of an $S$-orbit is given, the whole $S$--orbit inside $w$, and a fortiori, inside
 each return word $R \in \cR_w$, is fixed.

 Furthermore, each concatenated pair $RR'$ of return words induces a permutation
 $$
 \psi_{RR'}:\{ 1 , \dots, c\} \to \{ 1 , \dots, c\}
 $$
 by setting $\psi_{RR'}(j) = j'$ if the $S$-orbit passing through the $j$-th entry position in the prefix word $w$ of $Rw$ also passes through the $j'$-th entry position of the prefix word $w$ in $R'w$.
 Let $\gS$ denote the collection of permutations of $\{ 1 , \dots, c\}$.
 Number the $S$-orbits through $x$ by $k = 1, \dots, c$, and for $s \in \gS$, assume that the $i$-th orbit contains the $j = s(k)$-th entry position in return word $R_0$, then it also contains
 entry position\footnote{The traidtion composition of permutations is left-to-right, to the bracket $( s \cdot \psi_{RR'})$ indicate that we first compute the composition of permutations, and only then apply it to $k$.} $j' = (s \cdot \psi_{RR'})(k)$ in return word $R_0$.
 Thus the entries of the distinct orbits are described by a skew-product
 $$
 F:X_w \times \gS \to X_w \times \gS, \qquad
 (\, (R_i)_{i \in \Z} , s\, ) \mapsto (\ \sigma((R_i)_{i \in \Z}),\, s \cdot \psi_{R_0R_1} \, ).
 $$

\begin{example}
 We illustrate the construction of the group extension by an example of the Fibonacci substitution shift.
The fixed point of the
Fibonacci substitution shift in the Introduction
has fixed point
 \begin{equation}\label{eq:fib}
 \rho = 0\ \underbrace{10\ 01\ 0}_{R'}\underbrace{10\ 0}_{R}\underbrace{1001\ 0}_{R'}\underbrace{10010}_{R'}\underbrace{10\ 0}_{R}\underbrace{10010}_{R'}\underbrace{100}_{R}\underbrace{1001\ }_{R'}\dots
 \end{equation}
 This is a one-sided sequence, and $\sigma^{-1}(\rho)$ has two preimages
 in $\overline{\orb_\sigma(\rho)}$, but this issue disappears in the two-sided version.

 Let $w = 1001$, then there are two return words $R = 100$ and $R' = 10010$, see \eqref{eq:fib}. Assume the jump-function is $p \equiv 2$.
 Then $|w|$ is too short to precisely implement the above (in fact, one return word is even shorter than $w$ itself), but we can nonetheless declare $1$ and $2$ as the entry positions.
 Since $|R|$ and $|R'|$ are both odd, $\psi_{RR'} = \psi_{R'R} = \psi_{R'R'} = (12)$ always flips these positions.

 If $p \equiv 3$, then $\psi_{RR'} = e$, the identity, because $|R|=3$,
 but $\psi_{R'R} = \psi_{R'R'} = (123)$.
\end{example}

To continue the preparation for the proof of the main theorem, we state Lemma~\ref{lem:lr}, which basically says that the skew-product we constructed is linearly recurrent. We state it for  an arbitrary linearly recurrent shift $(X,\sigma)$ and skew-function $\psi:X \to \gS$, but will later apply it to
$(X_w,\sigma)$ and $\psi_{R_0R_1}$.

\begin{lemma}\label{lem:lr}
For every linearly recurrent subshift $(X,\sigma)$ with constant $L$ and continuous skew-function $\psi:X \to \gS$, the skew-product $F:(x,s) \mapsto (\sigma(x), \psi(x)) \cdot s$
is also linearly recurrent:
That is: there is $L^*$ depending only on $L$ such that for every $u \in \cL(X)$, $s \in \gS$
and $x \in [u]$, there is $1 \leq n \leq L^*|u|$ such that
  $F^n(x,s) \in [u] \times \{s\})$.
\end{lemma}

\begin{example}
 To give an idea of the difficulty of this lemma, consider a  subshift
 $(X,\sigma)$ on the alphabet $\{ A,B,C,D\}$ with transition graph as in Figure~\ref{fig:annoying_example}. The subsitutions at the arrows indicate the skew-function $\psi: X \to \gS$ for the skew-product $F(x,s) \mapsto (\sigma(x), s \cdot \psi_{x_0x_1})$. The subshift is not a subshift of finite type, so the transition graph only indicates which transitions between letters are possible, but which transitions are really allowed depends on other conditions, too.
 Start in the bottom-left vertex $A$ with initial permutation $s_0$.
 Then for every path returning to $(A,s_0)$ will need to pass by $C$ and $D$ an even number of times. This is something that cannot be derived from the properties of the two-letter words in $\cL(X)$ only. That is, also if $(X,\sigma)$ is linearly recurrent, it is not obvious that $(X \times \gS, F)$ is linearly recurrent, too.
 \end{example}

\begin{figure}[ht]
\begin{center}
\begin{tikzpicture}[scale=1.5]
\draw (0,0) circle (0.4cm); \node at (0,0) {A};
\draw (3,0) circle (0.4cm); \node at (3,0) {B};
\draw (6,0) circle (0.4cm); \node at (6,0) {C};
\draw (3,2) circle (0.4cm); \node at (3,2) {D};
\draw[->] (0.5,0.2) .. controls (1.5,0.5) .. (2.5,0.2); \node at (1.5,0.7) {(12)(34)};
\draw[<-] (0.5,-0.2) .. controls (1.5,-0.5) .. (2.5,-0.2);  \node at (1.5,-0.7) {(12)(34)};
\draw[->] (3.5,0.2) .. controls (4.5,0.5) .. (5.5,0.2);  \node at (4.5,0.7) {(12)};
\draw[<-] (3.5,-0.2) .. controls (4.5,-0.5) .. (5.5,-0.2);  \node at (4.5,-0.7) {id};
\draw[->] (3,0.5) -- (3,1.5);  \node at (3.4,1) {(34)};
\draw[->] (2.5,1.8) .. controls (0.5,1.2) .. (0.2,0.5);  \node at (1,1.6) {(12)};
\end{tikzpicture}
\caption{Transition graph of a non-SFT subshift with substitutions in $\gS$ at the edges. Extra restriction beyonds the transition graph exist.}
\label{fig:annoying_example}
\end{center}
\end{figure}

Before proving Lemma~\ref{lem:lr} formally, we give an outline of the proof.
First, we study the shapes of the transitive
components of $X \times \gS$. To find such component, choose $(x,s) \in X \times \gS$ and take its orbit closure, say $Y$. By minimality of $(X, \sigma)$, $Y$ projects to the whole of $X$ in the base direction, but in the fiber direction, for each $x \in X$, $Y$ projects to $\{ x \} \times s_x \gS_x$, where $\gS_x$ is a subgroup of $\gS$, the analog of $G_x$ in \eqref{eq:Gx}
and $s_x \in \gS$.
The assignment $x \mapsto \gS_x$ (and therefore $x \mapsto s_x$) is piecewise constant by Proposition~\ref{prop:cont}.
We can take $u \in \cL(X)$ so long that $[u] \subset U$ from the proof of Proposition~\ref{prop:cont}.

Therefore $s_x \equiv e$ and $\gS_x \equiv \gS'$ on $[u]$ for some subgroup $\gS'$ of $\gS$.
Replacing $(X,\sigma)$ by the first return map to $[u]$
(which is still uniquely ergodic, preserving a measure $\mu_u$),
we can assure that the skew-product on $(X_u, \gS')$ (that is, the shift with return words $\cR_u$ as alphabet) is transitive, and ergodic by Lemma~\ref{lem:ergod}.
Now Proposition~\ref{prop:Furst} implies that
this skew-product is also minimal, so for each $v \in \cL(X_u)$
and $s_* \in \gS'$, the set
$[v] \times \{ s_* \}$ reoccurs with a finite gap and has $\mu_u \times \nu_{\gS'}$-measure $\mu_u([v])/\#\gS' \geq 1/(L\, |v| \#\gS')$. Now a counting argument, going back to Boshernitzan, see \cite{BBL13} and \cite[Lemma 6.30]{Bruin23}, implies that
$[v] \times \{ s_0 \}$ reoccurs within $e^{L \#\gS'}|v|$ iterates.
\\[3mm]
\begin{proofof}{Lemma~\ref{lem:lr}}
Combining  Proposition~\ref{prop:cont} with Corollary~\ref{cor:Gxs_conj},
we see that every transitive $F$-invariant subset of $X_w \times \gS$
has the form $Y = \bigcup_{x \in X_w} \{ x \} \times s_x \gS_x$.

From Proposition~\ref{prop:cont},  we can take a cylinder set $[u]$ on which $s_x \equiv e$ and $\gS_x \equiv \gS'$ is a subgroup of $\gS$.
Therefore the first return map to $[u]$ is a subshift
$(X_u, \sigma)$ over the alphabet consisting of return words $\cR_u$.
It is linearly recurrent and has a transitive group extension
$$
F: X_u \times \gS' \to  X_{w} \times \gS',
\qquad (x, s) \mapsto (\sigma(x), s \cdot \psi(x))
$$.
Without loss of generality,m we can assume that $u$ is so long that the correspoding skew-function $\psi$
depends only on the first two coordinnates of $x$.

Let $\mu_u$ be the unique shift-invariant measure of $(X_u,\sigma)$ and $\nu_{\gS'}$ counting measure on $\gS'$.
By Lemma~\ref{lem:ergod}, $\mu_u \times \nu_{\gS'}|_{\gS'}$ is ergodic, and then
by Furstenberg's Proposition~\ref{prop:Furst},
$F: X_u \times \gS' \to X_u \times \gS'$ is also uniquely ergodic, and since the base map $\sigma:X_u \to X_u$ is minimal, $F$ is minimal as well.

View $F:X_u \times \gS' \to X_u \times \gS'$
as a subshift on the alphabet $\cR_u \times \gS'$.
For a word $x \in \cR_w^*$ and $s \in \gS'$, there is a unique word $\tilde x = \tilde x(x,s)$ in the alphabet $\cR_u \times \gS'$, such
$$
\tilde x_1 = (x_1,s)
\quad \text{ and } \quad
\tilde x_k = (\, x_k, s \cdot \psi(x_0x_1) \cdots \psi(x_{k-1}x_k ) \, )\text{ for }
2 \leq k \leq |x|.
$$
That is, $s$ together with $x_1 \dots x_k$ uniquely determine $\tilde x_k$.
Therefore
\begin{equation}\label{eq:us_0}
 \mu_u \times \nu_{\gS'}([v] \times \{ s \}) \geq \frac{1}{L |v| \# \gS'}
 \quad \text{ for the cylinder sets  }[v] \text{ in } X_u \text{ and } s \in \gS'.
\end{equation}

Now choose $v \in v_1\dots v_n$ in $\cL(X_{w})$ and let $N = N(v)$ denote
the length of the longest word $x = v_1 \dots v_n x_{n+1} \dots x_N$,
such that $F^{N}( [\tilde x(x,e)] ) \subset [v] \times \{ e \}$
and $F^k([\tilde x(x,e)] ) \not\subset [v] \times \{ e \}$
for $1 \leq k < N$. That is, the word $\tilde v = \tilde v(u,e)$ is prefix of $\tilde x$, but doesn't occur elsewhere in $\tilde x$.
Note that $N(v) < \infty$ by minimality of $F$.

Let $\tilde v_k = \tilde x_1 \dots \tilde x_k$.
Then~\eqref{eq:us_0} gives
$$
\mu_u \times \nu_{\gS'}([\tilde v_k]) = \frac{1}{\#\gS'} \mu_u([x_1 \dots x_k])
\geq \frac{1}{\#\gS'} \frac{1}{L \, k}.
$$
Because of the single occurrence of $\tilde v$ inside $\tilde x$, the cylinder sets
$\sigma^{N-k+1}([\tilde v_k])$ are pairwise disjoint.
Therefore
\begin{eqnarray*}
 1 &=& \mu_u \times \nu_{\gS'}(X_u \times \gS')
\geq \sum_{k=|v|}^{N(v)}  \mu_u \times \nu_{\gS'}(\sigma^{N-k+1}([\tilde v_k]) )  \\
&\geq& \sum_{k=|v|}^{N(v)} \frac{1}{L \#\gS'}\frac{1}{k} \geq
\frac{1}{L \# \gS'}\ \left( \log N(v) - \log|v|\right).
\end{eqnarray*}
Hence $N(v) \leq e^L\, |v|$, independently of $u$ and $s_* \in \gS'$.
Assuming the worst case that $\gS'$ is the entire permutation group, so $\#\gS = c!$,
and noting that $|R| \leq L |u|$ for each $R \in \cR_u$, we obtain the lemma with $L^* = e^{L c!} L|u|$.
\end{proofof}

\begin{remark}\label{rem:lr}
 The above proof uses no specific properties of the permutation group, except its finiteness.
 Therefore, Lemma~\ref{lem:lr} proves the more general statement:
 if $G$ is a finite group and $(X,\sigma)$ a linearly recurrent subshift, then the group extension $F:(x,g) \mapsto (\sigma(x), g \cdot \psi(x))$ is linearly recurrent, provided the skew-function is continuous.
\end{remark}

\begin{proofof}{Theorem~\ref{thm:main}}
Let $(X,\sigma)$ be a linearly recurrent subshift and $S$ a homeomorphic speedup.
The speedup $S$ is a shift over the alphabet
$\cB = \{b = x_0 \dots x_{p(x)-1}\}$ as in \eqref{eq:cB} and each allowed word in $\cB^*$ can be recoded into the $x \in X$ it came from by means of the substitution
$$
\chi: \cB \to \cA^*, \qquad b \mapsto  x_0 \dots x_{p(x)-1},
$$
with the usual extension $\chi:\cB^* \to \cA^*$ by concatenation.

To prove the first implication in Theorem~\ref{thm:main},
let $v \in \cB^*$ be some word for $S$.
Let $w$ be a sufficiently long word in $\cL(X)$ such that
$w_{K+1} \dots w_{|w|-K}$ contains $\chi(v)$, where $K$ comes from the construction of the group extension below Remark~\ref{rem:c+s}.
By linear recurrence, $w$ returns within $L |w|$ iterates, and by Lemma~\ref{lem:lr}, the combination $(w,s)$ reoccurs within $L^* |w|$ iterates.
But the same $s \in \gS$ means that the $S$-orbits also return in the same order, so the $S$-orbit starting with $v$ sees $v$ again.
Hence, $S$-pattern $v$ returns within
$L^*  |w| \leq 2 L^* p_{\max} |v|$
iterates of $S$.

Now for the reverse implication, assume that $(X,S)$ is linearly recurrent with constant $L$.
If $w \in \cL(X)$ is an allowed word for $(X,\sigma)$, then there is some minimal allowed word $v \in \cB^*$
such that $w$ is a subword of $\chi(v)$.
We have
$$
|v| \leq |w| \leq \chi(v) \leq p_{\max} |v|.
$$
Since $v$ returns with gap $\leq L|v|$,
$w$ returns with gap $\leq L p_{\max} |w|$.
\end{proofof}

{\bf Acknowledgements}
The author was supported by the ANR-FWF Project I 6750-N.
He also thanks Klaus Schmidt for useful discussions on non-abelian group extensions, and Jiao Changhua for communicating an issue with an earlier version of Lemma~\ref{lem:ergod}.

\end{document}